\documentclass
{amsart}
\usepackage{graphicx}
\newtheorem{thm}{Theorem}[section]
\theoremstyle{definition}
\newtheorem{cor}[thm]{Corollary}
\newtheorem{prop}[thm]{Proposition}
\newtheorem{defn}[thm]{Definition}

\newtheorem{rem}[thm]{Remark}
\newtheorem{ex}[thm]{Example}

\numberwithin{equation}{section}
\begin{document}
\title[$\psi $-second submodules of a module]{$\psi $-second submodules of a module}

\author{F. Farshadifar*}
\address{\llap{*\,} (Corresponding Author) Assistant Professor, Department of Mathematics, Farhangian University, Tehran, Iran.}
\email{f.farshadifar@cfu.ac.ir}

\author%
{H. Ansari-Toroghy**}

\newcommand{\acr}{\newline\indent}

\address{\llap{**\,}Department of pure Mathematics\\
Faculty of mathematical
Sciences\\
University of Guilan\\
P. O. Box 41335-19141, Rasht, Iran}
\email{ansari@guilan.ac.ir}

\begin{abstract}
Let  $R$ be a commutative ring with identity and $M$ be an $R$-module. Let $\psi : S(M)\rightarrow S(M) \cup \{\emptyset \}$  be a function, where $S(M)$ denote the set of all submodules of $M$.
The main purpose of this paper is to introduce and study the notion of $\psi$-second submodules of an $R$-module $M$.
\end{abstract}

\subjclass[2010]{13C05}%
\keywords {Second submodule, $\phi$-prime ideal, weak second submodule, $\psi $-second submodule}

\maketitle
\section{Introduction}
\noindent
Throughout this paper, $R$ will denote a commutative ring with
identity and $\Bbb Z$ will denote the ring of integers. We will denote the set of ideals of $R$ by $S(R)$ and the set of all submodules of $M$ by $S(M)$, where $M$ is an $R$-module.

Let $M$ be an $R$-module. A proper submodule $P$ of $M$ is said to be \emph{prime} if for any $r \in R$ and $m \in M$ with $rm \in P$, we have $m \in P$ or $r \in (P:_RM)$ \cite{Da78}. A non-zero submodule $N$ of $M$ is said to be \emph{second} if for each $a \in R$, the homomorphism $ N \stackrel {a} \rightarrow N$ is either surjective
or zero \cite{Y01}.

A non-zero submodule $S$ of an $R$-module $M$ is a \emph{weak second submodule} of $M$   if for each $r \in R$ and a submodule $K$ of $M$, $r \in (K:_RS) \setminus (K:_RM)$ implies that $S \subseteq K$ or $r \in Ann_R(S)$ \cite{FA18}.

Anderson and Bataineh in \cite{AB08} defined the notation of $\phi$-prime ideals as follows: let
$\phi: S(R)\rightarrow S(R) \cup \{\emptyset \}$ be a function. Then, a proper ideal $P$ of $R$ is \textit{$\phi$-prime} if for r$, s \in R$, $rs \in P \setminus \phi(P)$ implies that $r \in P$ or $s \in  P$.

Zamani in \cite{Za10} extended this concept to prime submodule. For a function $\phi : S(M)\rightarrow S(M) \cup \{\emptyset \}$, a proper submodule $N$ of $M$ is called \textit{$\phi$-prime} if whenever $r \in R$ and $x \in M$
with $rx \in N \setminus \phi(N)$, then $r \in (N :_R M)$ or $x \in N$.

Let $M$ be an $R$-module  and let $\psi: S(M) \rightarrow S(M) \cup \{\emptyset \}$ be a function.
The main purpose of this paper is to introduce and study the notion of $\psi$-second submodules of $M$ as a dual notion of  $\phi$-prime submodules of $M$.  We say that a non-zero submodule
$N$ of $M$ is a \textit{$\psi$-second submodule of $M$} if $r \in R$, $K$ a submodule of $M$, $rN\subseteq K$,  and $r\psi(N) \not \subseteq K $, then $N \subseteq K$ or $rN=0$. Among the other
results, we have shown that  if  $N$ is a $\psi$-second submodule of $M$ such that  $Ann_R(N) \psi(N) \not\subseteq N $, then $N$ is a second submodule of $M$ (see Theorem \ref{t2.3}).
We prove that if $H$ is a proper submodule of $M$ such that $(H:_RM)=0$, then $H$ is a second submodule of M if and only if $H$ is a
$\psi_1$-second submodule of $M$ (see Corollary \ref{c2.6}). In Theorem \ref{p92.8}, it is shown that if $\psi: S(M) \rightarrow S(M) \cup \{\emptyset \}$, $\phi: S(R) \rightarrow S(R) \cup \{\emptyset \}$ are  functions, then we have the following.
\begin{itemize}
 \item [(a)] If $S$ is a $\psi$-second submodule of $M$ such that $Ann_R(\psi(S))\subseteq \phi(Ann_R(S))$, then $Ann_R(S)$ is a $\phi$-prime ideal of $R$.
 \item [(b)] If $M$ is a comultiplication $R$-module, $S$ is a submodule of $M$ such that $\psi(S)=(0:_M\phi(Ann_R(S))$,  and $Ann_R(S)$ is a $\phi$-prime ideal of $R$, then $S$ is a $\psi$-second submodule of $M$.
\end{itemize}
Also, it is shown that if $a$ is an element of $R$ such that $(0 :_M a) \subseteq  a(0:_MaAnn_R((0:_Ma)))$ and  $(0:_Ma)$ is a $\psi_1$-second submodule of $M$, then  $(0:_Ma)$ is a second submodule of $M$ (see Theorem \ref{t2.12}). Moreover,  in Theorem \ref{t2.13}, we characterize
$\psi$-second submodules of $M$.

\section{Main results}
\noindent
\begin{defn}\label{d2.1}
Let $M$ be an $R$-module, $S(M)$ be the set of all
submodules of $M$,  and let $\psi: S(M) \rightarrow S(M) \cup \{\emptyset \}$ be a function. We say that a  non-zero submodule
$N$ of $M$ is a \textit{$\psi$-second submodule of $M$} if $r \in R$, $K$ a submodule of $M$, $rN\subseteq K$,  and $r\psi(N) \not \subseteq K $, then $N \subseteq K$ or $rN=0$.
\end{defn}

We use the following functions $\psi: S(M) \rightarrow S(M) \cup \{\emptyset \}$.
$$\psi_{M}(N)=M, \qquad \forall N \in S(M),$$
$$\psi_i(N)=(N:_MAnn_R^i(N)), \ \forall N \in S(M), \ \forall i \in \Bbb N,$$
$$\psi_\sigma(N)=\sum^{\infty}_{i=1}\psi_i(N), \qquad \forall N \in S(M).$$
Then it is clear that $\psi_M$-second submodules are weak second submodules. Clearly,  for any submodule and every positive integer $n$, we have the
following implications:
$$
second \Rightarrow \psi_{n-1}-second \Rightarrow \psi_{n}-second \Rightarrow \psi_\sigma-second.
$$
For functions $\psi, \theta: S(M) \rightarrow S(M) \cup \{\emptyset \}$, we write $\psi \leq \theta $  if $\psi(N) \subseteq \theta (N)$ for each $N \in S(M)$. So whenever $\psi \leq \theta$, any $\psi$-second submodule is $\theta$-second.

\begin{thm}\label{t1.1}
\cite[2.10]{AF12}. For a submodule S of an R-module M the following statements
are equivalent.
\begin{itemize}
\item [(a)] $S$ is a second submodule of $M$.
\item [(b)] $S \not= 0$ and $rS \subseteq K$, where $r \in R$ and $K$ is a submodule of $M$, implies
either $rS = 0$ or $S \subseteq K$.
\end{itemize}
\end{thm}

\begin{thm}\label{t2.3}
Let $M$ be an R-module and  $\psi: S(M) \rightarrow S(M) \cup \{\emptyset \}$ be a function.
Let $N$ be a $\psi$-second submodule of $M$ such that  $Ann_R(N)\psi(N)\not\subseteq N $. Then $N$ is a second submodule of $M$.
\end{thm}
\begin{proof}
Let $a \in R$ and $K$ be a submodule of $M$ such that $aN \subseteq K$. If $a\psi(N) \not \subseteq K$, then we are done because $N$ is a $\psi$-second submodule of $M$. Thus suppose that $a\psi(N)\subseteq K$. If $a\psi(N) \not \subseteq N$, then $a\psi(N) \not \subseteq N \cap K$. Hence $aN \subseteq N \cap K$ implies that $N \subseteq N \cap K \subseteq K$ or $aN=0$ as needed. So let $a\psi(N) \subseteq N$. If $Ann_R(N)\psi(N) \not \subseteq K$, then $(a+Ann_R(N))\psi(N) \not \subseteq K$. Thus $(a+Ann_R(N))N \subseteq K$ implies that $N \subseteq K$ or $aN=(a+Ann_R(N))N=0$, as required. So let $Ann_R(N)\psi(N)\subseteq K$. Since  $Ann_R(N)\psi(N) \not\subseteq N $, there exists $ b \in Ann_R(N)$ such that $b\psi(N) \not \subseteq N$.  Hence and $b\psi(N)\not \subseteq N\cap K $. This in turn implies that $(a+b)\psi(N) \not \subseteq N \cap K$. Thus $(a+b)N\subseteq N\cap K $ implies that $N\subseteq N \cap K\subseteq K$ or $(a+b)N=aN=0$ as needed.
\end{proof}

\begin{cor}\label{c2.4}
Let $N$ be a weak second submodule of an $R$-module $M$ such that $Ann_R(N)M \not\subseteq N$.
Then $N$ is a second submodule of $M$.
\end{cor}
\begin{proof}
In the Theorem \ref{t2.3} set  $\psi =\psi_M$.
\end{proof}

\begin{cor}\label{c42.3}
Let $M$ be an R-module and  $\psi: S(M) \rightarrow S(M) \cup \{\emptyset \}$ be a function.
If $N$ is a $\psi$-second submodule of $M$ such that  $(N:_MAnn^2_R(N))\subseteq\psi(N)$, then $N$ is a $\psi_\sigma$-second submodule of $M$.
\end{cor}
\begin{proof}
If $N$ is a second submodule of $M$, then the result is clear. So suppose that
$N$ is not a second submodule of $M$. Then by Theorem \ref{t2.3},  we have  $Ann_R(N)\psi(N) \subseteq N $.  Therefore, by assumption,
$$
(N:_MAnn^2_R(N)) \subseteq \psi(N)\subseteq (N:_MAnn_R(N)).
$$
This implies that $\psi(N)=(N:_MAnn_R^2(N))=(N:_MAnn_R(N))$ because always $(N:_MAnn_R(N)) \subseteq (N:_MAnn_R^2(N))$.
Now
$$
(N:_MAnn_R^3(N))=((N:_MAnn_R^2(N)):_MAnn_R(N))=
$$
$$
((N:_MAnn_R(N)):_MAnn_R(N))=(N:_MAnn_R^2(N))=\psi(N).
$$
By continuing, we get that $\psi(N)=(N:_MAnn_R^i(N))$ for all $i \geq 1$. Therefore, $\psi(N)=\psi_\sigma(N)$ as needed.
\end{proof}

\begin{thm}\label{t2.5}
Let $M$ be an R-module and $\psi: S(M) \rightarrow S(M) \cup \{\emptyset \}$ be a function.
Let $H$ be a submodule of $M$ such that far all ideals $I$ and $J$ of $R$, $(H:_MI)\subseteq (H:_MJ)$ implies that $J \subseteq I$. If $H$ is not a second submodule of $M$, then $H$ is not a $\psi_1$-second submodule of $M$.
\end{thm}
\begin{proof}
As $H$ is not a second submodule of $M$,  there exists $r \in R$ and a submodule $K$ of $M$ such that
$rH \not=0$ and $H \not \subseteq K$, but $rH \subseteq K$ by Theorem \ref{t1.1}.
We have $H \not \subseteq K \cap H$ and $rH \subseteq K \cap H$. If $r(H:_MAnn_R(H)) \not \subseteq K \cap H $, then by our definition $H$ is not a $\psi_1$-second submodule of $M$.  So let $r(H:_MAnn_R(H)) \subseteq K \cap H $. Then $r(H:_MAnn_R(H)) \subseteq K \cap H \subseteq H$. Thus $(H:_MAnn_R(H)) \subseteq (H:_Mr)$ and so by assumption,  $r\in Ann_R(H)$. This is a contradiction.
\end{proof}

\begin{cor}\label{c2.6}
Let $M$ be an R-module and $\psi: S(M) \rightarrow S(M) \cup \{\emptyset \}$ be a function.
Let $H$ be a submodule of $M$ such that far all ideals $I$ and $J$ of $R$, $(H:_MI)\subseteq (H:_MJ)$ implies that $J \subseteq I$. Then $H$ is a second submodule of $M$ if and only if $H$ is a
$\psi_1$-second submodule of $M$.
\end{cor}

An $R$-module $M$ is said to be a \emph{multiplication module} if for every submodule $N$ of $M$, there exists an ideal $I$ of $R$ such that $N=IM$ \cite{Ba81}. It is easy to see that $M$ is a multiplication module if and only if $N=(N:_RM)M$ for each submodule $N$ of $M$.

\begin{thm}\label{t2.7}
Let $M$ be an R-module,  $\phi: S(R) \rightarrow S(R) \cup \{\emptyset \}$, and $\chi: S(M) \rightarrow S(M) \cup \{\emptyset \}$ be  functions such that $\chi(P)=\phi((P:_RM))M$.
\begin{itemize}
  \item [(a)] If $P$ is a $\chi$-prime submodule of $M$ such that $(\chi(P):_RM)\subseteq \phi((P:_RM))$, then $(P:_RM)$ is a $\phi$-prime ideal of $R$.
  \item [(b)] If $M$ is a multiplication $R$-module and $(P:_RM)$ is a $\phi$-prime ideal of $R$, then $P$ is a $\chi$-prime submodule of $M$.
\end{itemize}
\end{thm}
\begin{proof}
(a) Let $ab \in(P:_RM)\setminus  \phi ((P:_RM)$ for some $a, b \in R$.
If $abM\subseteq \chi(P))$, then $ab \in \phi((P:_RM))$, a contradiction. Thus  $abM\not\subseteq \chi(P)$. Therefore,  $aM \subseteq P$ or $bM \subseteq P$ because $P$ is a $\chi$-prime submodule of $M$.

(b) Let $ax \in P \setminus \chi(P)=P\setminus \phi((P:_RM))M$. Then $a(Rx:_RM)M \subseteq P$. If $a(Rx:_RM) \subseteq \phi ((P:_RM))$, then $a(Rx:_RM)M \subseteq \phi ((P:_RM))M$. As $M$ is a multiplication $R$-module, we have $ax \in Rx=(Rx:_RM)M$. Therefore, $ax \in \phi ((P:_RM))M$, a contradiction. Thus $a(Rx:_RM)\not \subseteq \phi ((P:_RM))$ and so by assumption, $a\in (P:_RM)$ or $(Rx:_RM) \subseteq (P:_RM)$ as needed.
\end{proof}

\begin{thm}\label{p92.8}
Let $M$ be an R-module and  $\psi: S(M) \rightarrow S(M) \cup \{\emptyset \}$, $\phi: S(R) \rightarrow S(R) \cup \{\emptyset \}$ be  functions.
\begin{itemize}
  \item [(a)] If $S$ is a $\psi$-second submodule of $M$ such that $Ann_R(\psi(S))\subseteq \phi(Ann_R(S))$, then $Ann_R(S)$ is a $\phi$-prime ideal of $R$.
 \item [(b)] If $M$ is a comultiplication $R$-module, $S$ is a submodule of $M$ such that $\psi(S)=(0:_M\phi(Ann_R(S))$,  and $Ann_R(S)$ is a $\phi$-prime ideal of $R$, then $S$ is a $\psi$-second submodule of $M$.
\end{itemize}
 \end{thm}
\begin{proof}
(a) Let $ab \in Ann_R(S)\setminus  \phi (Ann_R(S))$ for some $a, b \in R$. Then $ab \psi (S) \not =0$ by assumption.
If $a\psi(S)\subseteq (0:_Mb) $, then $ab\psi(S)=0$, a contradiction. Thus $a\psi(S)\not \subseteq (0:_Mb)$. Therefore, $S \subseteq (0:_Mb)$ or $aS=0$ because $S$ is a $\psi$-second submodule of $M$.

(b) Let $a \in R$ and $K$ be a submodule of $M$ such that $aS \subseteq K$ and $a\psi(S) \not \subseteq K$.
As $aS \subseteq K$, we have $S \subseteq (K:_Ma)$.  It follows that
$$
S \subseteq ((0:_MAnn_R(K)):_Ma)=(0:_MaAnn_R(K)).
$$
This implies that $aAnn_R(K) \subseteq Ann_R((0:_MaAnn_R(K))) \subseteq Ann_R(S)$. Hence, $aAnn_R(K)\subseteq Ann_R(S)$. If $aAnn_R(K)\subseteq \phi( Ann_R(S))$, then
$$
\psi(S)=(0:_M\phi(Ann_R(S)) \subseteq ((0:_MAnn_R(K):_Ma).
$$
As $M$ is a comultiplication $R$-module, we have  $a\psi(S)\subseteq K$, a contradiction. Thus $aAnn_R(K) \not\subseteq  \phi(Ann_R(S))$ and so as $Ann_R(S)$ is a $\phi$-prime ideal of $R$, we conclude that $aS=0$ or
$$
S=(0:_MAnn_R(S)) \subseteq (0:_MAnn_R(K))=K,
$$
as needed.
\end{proof}

The following example shows that the condition ``$M$ is a comultiplication $R$-module" in   Theorem \ref{p92.8} (b) can not be omitted.
\begin{ex}\label{e22.14}
Let $R= \Bbb Z$, $M=\Bbb Z\oplus \Bbb Z$, and $S=2\Bbb Z \oplus 2\Bbb Z$. Clearly,  $M$ is not a comultiplication $R$-module. Suppose that  $\phi: S(R) \rightarrow S(R) \cup \{\emptyset \}$ and $\psi: S(M) \rightarrow S(M) \cup \{\emptyset \}$ be  functions such that $\phi(I)=I$ for each ideal $I$ of $R$ and  $\psi(S)=M$. Then clearly,  $Ann_R(S)=0$ is a  $\phi$-prime ideal of $R$ and  $\psi(S)=M=(0:_M\phi(Ann_R(S))$. But as
 $3S \subseteq 6\Bbb Z \oplus  6\Bbb Z$,  $S \not\subseteq 6\Bbb Z \oplus  6\Bbb Z$,  and $3S\not=0$, we have that $S$ is not a $\psi$-second submodule of $M$.
\end{ex}

\begin{prop}\label{p2.9}
Let $M$ be an R-module,  $\psi: S(M) \rightarrow S(M) \cup \{\emptyset \}$  be a function, and  $N$ be a $\psi$-second
submodule of $M$. Then we have the following statements.
\begin{itemize}
  \item [(a)] If $K$ is a submodule of $M$ with $K \subset N$ and  $\psi_K: S(M/K) \rightarrow S(M/K) \cup \{\emptyset \}$ be a function such that $ \psi_K(N/K)= \psi(N)/K$, then $N/K$ is a $ \psi_K$-second submodule of $M/K$.
  \item [(b)] Let $N$ be a finitely generated submodule of $M$, $S$ be a  multiplicatively closed subset of $R$ with $Ann_R(N) \cap S=\emptyset$, and  $S^{-1}\psi: S(S^{-1}M) \rightarrow S(S^{-1}M) \cup \{\emptyset \}$ be a function such that $(S^{-1}\psi)(S^{-1}N)= S^{-1}\psi(N)$. Then $S^{-1}N$ is a $S^{-1}\psi$-second submodule of $S^{-1}M$.
\end{itemize}
\end{prop}
\begin{proof}
These are straightforward.
\end{proof}

\begin{prop}\label{p2.10}
Let $M$ and $\acute{M}$ be $R$-modules and  $f : M\rightarrow \acute{M}$ be an $R$-monomorphism. Let  $\psi: S(M) \rightarrow S(M) \cup \{\emptyset \}$ and $\acute{\psi}: S(\acute{M}) \rightarrow S(\acute{M}) \cup \{\emptyset \}$  be functions such that $\psi(f^{-1}(\acute{N}))=f^{-1}(\acute{\psi}(\acute {N}))$, for each submodule $\acute {N}$ of $\acute{M}$. If
 $\acute{N}$ is a $\acute{\psi}$-second submodule of $\acute{M}$ such that $\acute{N} \subseteq Im(f)$, then $f^{-1}(\acute{N})$ is a  $\psi$-second submodule of $M$.
\end{prop}
\begin{proof}
As $\acute{N} \not =0$ and $\acute{N} \subseteq Im(f)$, we have $f^{-1}(\acute{N})\not =0$. Let
$a \in R$ and $K$ be a submodule of $M$ such that $af^{-1}(\acute{N}) \subseteq K$ and $a\psi(f^{-1}(\acute{N})) \not \subseteq K$. Then by using assumptions, $a\acute{N} \subseteq f(K)$ and
 $a\acute{\psi}(\acute{N}) \not \subseteq f(K)$. Thus $a\acute{N}=0$ or $\acute{N}\subseteq f(K)$. This implies that  $af^{-1}(\acute{N})=0$  or $f^{-1}(\acute{N}) \subseteq K$
as needed.
\end{proof}

A proper submodule $N$ of
$M$ is said to be \emph{completely irreducible} if $N=\bigcap _
{i \in I}N_i$, where $ \{ N_i \}_{i \in I}$ is a family of
submodules of $M$, implies that $N=N_i$ for some $i \in I$. It is
easy to see that every submodule of $M$ is an intersection of
completely irreducible submodules of $M$ \cite{FHo06}.

\begin{rem}\label{r22.2}
Let $N$ and $K$ be two submodules of an $R$-module $M$. To prove $N\subseteq K$, it is enough to show that if $L$ is a completely irreducible submodule of $M$ such that $K\subseteq L$, then $N\subseteq L$.
\end{rem}

\begin{prop}\label{p2.11}
Let $M$ be an R-module,  $\psi: S(M) \rightarrow S(M) \cup \{\emptyset \}$  be a function, and  let $N$ be a $\psi_1$-second
submodule of $M$. Then we have the following statements.
\begin{itemize}
\item [(a)] If for $a \in R$, $aN\not=N$, then $(N:_MAnn_R(N))\subseteq (N:_Ma)$.
\item [(b)] If $J$ is an ideal of $R$ such that $Ann_R(N)\subseteq J$ and $JN \not=N$, then
$(N:_MAnn_R(N))=(N:_MJ)$.
\end{itemize}
\end{prop}
\begin{proof}
(a) By Remark \ref{r22.2}, there exists a completely irreducible submodule $L$ of $M$ such that $aN\subseteq L$ and $N \not \subseteq L$.
If $aN=0$, then clearly $(N:_MAnn_R(N))\subseteq (N:_Ma)$. So let  $aN\not=0$. Since $N$ is a $\psi_1$-second
submodule of $M$, we must have $a(N:_MAnn_R(N)) \subseteq L$. Now let $\acute{L}$ be a completely irreducible submodule of $M$ such that $N \subseteq \acute{L}$. Then  $N \not\subseteq \acute{L}\cap L$ and  $aN \subseteq \acute{L}\cap L$. Hence as $N$ is a $\psi_1$-second submodule of $M$, we have $a(N:_MAnn_R(N)) \subseteq \acute{L}\cap L $. Thus $a(N:_MAnn_R(N)) \subseteq \acute{L}$. Therefore, $a(N:_MAnn_R(N))\subseteq N$ by Remark \ref{r22.2}. It follows that $(N:_MAnn_R(N))\subseteq (N:_Ma)$.

(b) This follows from part (a).
\end{proof}

\begin{thm}\label{t2.12}
Let $M$ be an $R$-module,  $\psi: S(M) \rightarrow S(M) \cup \{\emptyset \}$  be a function, and let $a$ be an element of $R$ such that $(0 :_M a) \subseteq  a(0:_MaAnn_R((0:_Ma)))$. If $(0:_Ma)$ is a $\psi_1$-second submodule of $M$, then  $(0:_Ma)$ is a second submodule of $M$.
\end{thm}
\begin{proof}
Let $N:=(0:_Ma)$ be a $\psi_1$-second submodule of $M$. Then  $(0:_Ma)\not=0$. Now let $t \in  R$ and  $K$ be a submodule of $M$ such that
$t(0:_Ma) \subseteq K$.  If $t(N:_MAnn_R(N)) \not \subseteq K$,
then $t(0:_Ma)=0$ or $(0:_Ma) \subseteq K$ since $(0:_Ma)$ is a $\psi_1$-second submodule of $M$. So suppose
that $t(N:_MAnn_R(N)) \subseteq K$. Now we have $(t+a)(0:_Ma) \subseteq K$. If $(t+a)(N:_MAnn_R(N)) \not \subseteq K $, then as $(0:_Ma)$ is a $\psi_1$-second submodule of $M$,  $(t+a)(0:_Ma)=0$ or $(0:_Ma) \subseteq K$ and we are done. So assume that $(t+a)(N:_MAnn_R(N))  \subseteq K $.
Then $t (N:_MAnn_R(N)) \subseteq K$ gives that $a(N:_MAnn_R(N)) \subseteq K$. Hence by assumption, $(0 :_M a) \subseteq K$ and the result follows from Theorem \ref{t1.1}.
\end{proof}

\begin{thm}\label{t2.13}
Let $N$ be a non-zero submodule of an $R$-module $M$ and  $\psi: S(M) \rightarrow S(M) \cup \{\emptyset \}$  be a function.
Then the following are equivalent:
\begin{itemize}
\item [(a)] $N$ is a $\psi$-second submodule of $M$;
\item [(b)] for completely irreducible submodule $L$ of $M$ with $N \not \subseteq L$, we have $(L:_RN)=Ann_R(N) \cup (L:_R\psi(N))$;
\item [(c)] for completely irreducible submodule $L$ of $M$ with $N \not \subseteq L$, we have $(L:_RN)=Ann_R(N)$ or $(L:_RN)=(L:_R\psi(N))$;
\item [(d)] for any ideal $I$ of $R$ and any submodule $K$ of $M$, if $IN \subseteq K$ and $I\psi(N) \not \subseteq K$,
then $IN=0$ or $N \subseteq K$.
\item [(e)] for each $a \in R$ with $a\psi(N) \not \subseteq aN$, we have $aN=N$ or $aN=0$.
\end{itemize}
\end{thm}
\begin{proof}
$(a)\Rightarrow (b)$. Let for a completely irreducible submodule $L$ of $M$ with $N \not \subseteq L$, we have $a \in(L:_RN)\setminus (L:_R\psi(N))$. Then $a\psi(N)\not \subseteq  L$.
Since $N$ is a $\psi$-second submodule of $M$, we have $a \in Ann_R(N)$. As we may assume that
$\psi(N) \subseteq N$, the other inclusion always holds.

$(b)\Rightarrow (c)$. This follows from the fact that if a subgroup is a union of two subgroups, it is equal to one of them.

$(c)\Rightarrow (d)$.
Let $I$ be an ideal of $R$ and $K$ be a submodule of $M$ such that $IN\subseteq K$ and $I\psi(N) \not \subseteq K$.
Suppose $I \not \subseteq Ann_R(N)$ and $N\not\subseteq K$. We show that $I\psi(N) \subseteq K$. Let $a \in I$ and
$L$ is a completely irreducible submodule of $M$ with $K \subseteq L$. First let $a\not \in Ann_R(N)$. Then, since $aN \subseteq L$, we have $(L:_RN)\not=Ann_R(N)$.
Hence by our assumption $(L:_RN)=(L:_R\psi(N))$. So $a\psi(N) \subseteq L$. Now assume that
$a \in I \cap Ann_R(N)$. Let $u \in I \setminus Ann_R(N)$. Then $a + u \in I \setminus Ann_R(N)$. So by the
first case, for each completely irreducible submodule $L$ of $M$ with $K \subseteq L$ we have $u\psi(N) \subseteq L$ and $(u+a)\psi(N) \subseteq L$. This gives that
$a\psi(N) \subseteq L$. Thus in any case $a\psi(N) \subseteq L$. Thus $I\psi(N) \subseteq L$. Therefore $I\psi(N) \subseteq K$ by Remark \ref{r22.2}.

$(d)\Rightarrow (a)$.
This is clear.

$(a)\Rightarrow (e)$.
Let $a \in R$ such that $a\psi(N) \not \subseteq aN$. Then $aN \subseteq aN$ implies that $N \subseteq aN$ or $aN=0$ by part (a). Thus $N =aN$ or $aN=0$, as requested.

$(e)\Rightarrow (a)$.
Let $a \in R$ and $K$ be a submodule of $M$ such that $aN \subseteq K$ and $a\psi(N) \not \subseteq K$. If $a\psi(N) \subseteq aN$, then $aN \subseteq K$ implies that $a\psi(N) \subseteq K$, a contradiction. Thus by part (e), $aN=N$ or $aN=0$. Therefore,  $N \subseteq K$ or $aN=0$, as needed.
\end{proof}

\begin{ex}\label{e2.114}
Let $N$ be a non-zero submodule of an $R$-module $M$ and  let $\psi: S(M) \rightarrow S(M) \cup \{\emptyset \}$  be a function.
If $\psi(N)=N$, then  $N$ is a $\psi$-second submodule of $M$ by Theorem \ref{t2.13} $(e)\Rightarrow (a)$.
\end{ex}

Let $R_1$ and $R_2$ be two commutative rings with identity. Let $M_1$ and $M_2$ be $R_1$ and
$R_2$-module, respectively and put $R = R_1 \times R_2$. Then $M = M_1 \times M_2$ is an $R$-module
and each submodule of $M$ is of the form $N = N_1 \times N_2$ for some submodules $N_1$ of
$M_1$ and $N_2$ of $M_2$. Suppose that $\psi^i: S(M_i) \rightarrow S(M_i) \cup \{\emptyset \}$  be a function for $i=1, 2$. The second submodules of the $R = R_1 \times R_2$-module $M = M_1 \times M_2$ are in the form $S_1 \times 0$ or $0\times S_2$, where $S_1$ is a second
submodule of $M_1$ and $S_2$ is a second submodule of $M_2$ \cite[2.23]{AF1112}.  The following example, shows that this is not true for correspondence
$\psi^1 \times \psi^2$-second submodules in general.

\begin{ex}\label{e2.14}
Let $R_1 = R_2 = M_1 = M_2 =S_1= \Bbb Z_6$. Then clearly,  $S_1$  is
a weak second submodule of $M_1$. However,
$
(\bar{2},\bar{1})(\Bbb Z_6 \times 0) \subseteq \bar{2}\Bbb Z_6 \times \bar{3}\Bbb Z_6
$
and $(\bar{2},\bar{1})(\Bbb Z_6 \times \Bbb Z_6) \not \subseteq \bar{2}\Bbb Z_6 \times \bar{3}\Bbb Z_6$. But
$(\bar{2}, \bar{1})(\Bbb Z_6 \times 0) =\bar{2}\Bbb Z_6\times 0 \not =0\times 0$, and $\Bbb Z_6 \times 0 \not \subseteq \bar{2}\Bbb Z_6 \times \bar{3}\Bbb Z_6$.
Therefore,  $S_1 \times 0$ is not a weak second submodule of $ M_1 \times M_2$.
\end{ex}

\begin{thm}\label{t2.133}
Let $R = R_1 \times R_2$ be a ring and $M = M_1 \times M_2$
be an $R$-module, where $M_1$ is an $R_1$-module and $M_2$ is an $R_2$-module. Suppose that $\psi^i: S(M_i) \rightarrow S(M_i) \cup \{\emptyset \}$  be a function for $i=1, 2$. Then  $S_1 \times 0$ is a $\psi^1 \times \psi^2$-second submodule of $M$,  where $S_1$ is a  $\psi^1$-second submodule of $M_1$ and  $\psi^2(0)=0$.
\end{thm}
\begin{proof}
Let $(r_1, r_2) \in R$ and $K_1 \times K_2$ be a submodule of $M$
such that $(r_1, r_2)(S_1 \times 0) \subseteq K_1 \times K_2$ and
$$
(r_1, r_2) ((\psi^1 \times \psi^2)(S_1 \times 0))=r_1\psi^1(S_1)\times r_2\psi^2(0)=r_1\psi^1(S_1)\times 0\not\subseteq K_1 \times K_2
$$
Then
 $r_1S_1 \subseteq K_1 $ and $r_1 \psi^1(S_1) \not\subseteq K_1$.
Hence,
$r_1S_1=0$ or $S_1 \subseteq K_1 $ since $S_1$ is a  $\psi^1$-second submodule of $M_1$.
Therefore, $(r_1, r_2)(S_1 \times 0)=0 \times 0$ or $S_1 \times 0 \subseteq K_1 \times K_2$, as requested.
 \end{proof}

 \bibliographystyle{amsplain}

\end{document}